\documentclass[leqno,a4paper]{article}
\usepackage{amsmath,bbm, amsthm,amssymb}
\usepackage{eufrak}

\newcommand{\mK}{\mathcal{K}}

\newcommand{\be}{\boldsymbol{e}}
\newcommand{\bu}{\boldsymbol{u}}

\newcommand{\ff}{\boldsymbol{f}}

\newcommand{\bC}{\boldsymbol{C}}

\newcommand{\bL}{\boldsymbol{L}}
\newcommand{\bS}{\boldsymbol{S}}
\newcommand{\bT}{\boldsymbol{T}}

\newcommand{\R}{{\mathbb R}}

\newcommand{\ox}{{\overline{x}}}
\newcommand{\oy}{{\overline{y}}}

\newcommand{\na}{{\nabla}}
\newcommand{\pa}{{\partial}}
\newcommand{\al}{{\alpha}}
\newcommand{\te}{{\theta}}
\newcommand{\bt}{{\beta}}

\newcommand{\ro}{{\rho}}

\newcommand{\la}{{\lambda}}
\newcommand{\de}{{\delta}}
\newcommand{\om}{{\omega}}
\newcommand{\sg}{{\sigma}}
\newcommand{\z}{{\zeta}}

\newcommand{\Om}{{\Omega}}
\newcommand{\De}{{\Delta}}
\newcommand{\Ov}{{\overline{\Omega}}}
\newcommand{\Ga}{{\Gamma}}

\newcommand{\ph}{{\phi}}
\newcommand{\cf}{{\widetilde{f}}}

\def\vs1{\vspace{1ex}}
\newtheorem{definition}{Definition}[section]
\newtheorem{theorem}{Theorem}[section]

\newtheorem{corollary}{Corollary}[section]
\newtheorem{remark}{Remark}[section]

\newtheorem{acknowledgement}{Acknowledgement}[section]
\numberwithin{equation}{section}
\def\be{\begin{equation}}
\def\ee{\end{equation}}
\def\ed{\end{document}}
\begin{document}
\title{\bf\normalsize H-log spaces of continuous functions, potentials, and
elliptic boundary value problems.}
\author{by H.~Beir\~ao da Veiga}
\maketitle
\begin{abstract}
In these notes we study a family of Banach spaces, denoted $\,
D^{0,\,\al}(\Ov)\,,$ $\,\al \in\,\R^+\,,$ and called H-log spaces.
For $\,0<\,\la\leq\,1\,,$ one has $ C^{0,\,\la}(\Ov)\subset
D^{0,\,\al}(\Ov) \subset\,C(\Ov)\,,$ with compact embedding. These
spaces present the following "intermediate" regularity behavior.
Solutions $\,u\,$ of second order linear elliptic boundary value
problems, under "external forces" $\,f\in\, D^{0,\,\al}(\Ov)\,$ for
some $\,\al>\,1\,,$ satisfy $\,\na^2\,u\in\,
D^{0,\,\al-\,1}(\Ov)\,$. This result is optimal, since
$\,\na^2\,u\in\, D^{0,\,\beta}(\Ov)\,,$  for some $\,\beta
>\,\al-1\,,$ is false in general. We present a preliminary study on this subject.

\vspace{.2cm}

\,35B65, \,35J25\,.

\vspace{.2cm}

\end{abstract}
\bibliographystyle{amsplain}
\section{Motivation.}\label{motives}
Assume that one looks for subspaces $\,X(\Ov)\,$ of $\,C(\Ov)\,,$
\emph{as large as possible}, such that solutions $\,u\,$ of second
order linear elliptic boundary value problems, under "external
forces" $\,f \in X(\Ov)\,,$ have second order continuous derivatives
$\,\na^2\,u\,$ \emph{up to the boundary}. It is well known that
$\,X(\Ov)=\,C(\Ov)\,$ is too large. On the other hand, the classical
H\"older spaces $\,C^{0,\,\la}(\Ov)\,$ are too narrow. In reference
\cite{BVJDE}, dedicated to the Euler two dimensional evolution
equations, we have appealed to a Banach space $\,C_*(\,\Ov)\,$,
suitable for the above purposes. If $\,f\in\,C_*(\,\Ov)\,$ then
$\,\na^2\,u\in\,C(\,\Ov)\,.$ However, full regularity $\,\na^2\,u
\in\,C_*(\,\Ov)\,$ is unreachable. This leads to the following
question. For generical data $\,f\in\,C_*(\,\Ov)\,$, is there some
additional, significant regularity for $\,\na^2\,u \,,$ besides
continuity? Recall that, in the framework of H\"older spaces, full
regularity is restored, since $\,\na^2\,u \,$ and $\,f\,$ have the
same regularity. This different behavior leads us to look for
intermediate situations. In these notes we define and study a family
of Banach spaces, denoted here $\, D^{0,\,\al}(\Ov)\,,$
$\,\al>\,0\,,$ and called H-log spaces. For $\,\al>\,1>\,\bt>\,0\,,$
and $\,0<\,\la\leq\,1\,,$ one has $ C^{0,\,\la}(\Ov)\subset
D^{0,\,\al}(\Ov) \subset \,C_*(\Ov) \subset
D^{0,\,\beta}(\Ov)\subset\,C(\Ov)\,,$ with compact embeddings. These
spaces present the following "intermediate" regularity behavior. If
$\,f\in\, D^{0,\,\al}(\Ov)\,$ for some $\,\al>\,1\,,$ then
$\,\na^2\,u\in\, D^{0,\,\al-\,1}(\Ov)\,$. This result is optimal. We
present here a preliminary study on this subject. For some
information on related references see the acknowledgment, at the end
of section \ref{one}.

\section{Results.}\label{introd}
In the following $\Om$ is an open, bounded, connected set in
$\R^n\,$, locally situated on one side of its boundary $\,\Ga\,.$
The boundary $\,\Ga\,$ is of class $\,C^{2,\,\la}\,,$ for some
$\,\la\,,$ $\,0<\,\la \leq \,1\,.$\par%
By $\,C(\Ov)\,$ we denote the Banach space of all real continuous
functions $\,f\,$ defined in $\,\Ov\,$. The classical "sup" norm is
denoted by $ \|\,f\,\|\,. $ We also appeal to the classical spaces
$\,C^k(\Ov)\,$ endowed with their usual norms $ \|\,u\,\|_k\,,$ and
to the H\"older spaces $\,C^{0,\,\la}(\Ov)\,,$ endowed with the
standard semi-norms and norms, denoted here by the non-standard
symbols
$$
[\,f\,]_{H(\la)}\,, \quad \textrm{and} \quad \|\,f\,\|_{H(\la)}\,.
$$
The meaning of $\,\|\,f\,\|_{H(k,\,\la)}\,$ is obvious. Further,
$\,C^{\infty}(\Ov)\,$ denotes the set of all restrictions to $\Ov\,$
of indefinitely differentiable functions in $\R^n\,$. Norms in
function spaces, whose elements are vector fields, are defined in
the usual way by means of the corresponding norms of the
components.\par%
The symbol $c\,$ denotes generical positive constants depending at
most on $\,n\,$ and $\Om\,.$ The symbol $C\,$ denotes positive
constants depending at most on $\,n\,,$ $\Om\,,$ and $\,\al\,,$ see
below. We may use the same symbol to denote different constants.
$\,X\Subset\,Y\,$ means that the immersion is compact.

\vspace{0.2cm}

For an arbitrary, stationary or evolution, boundary value problem we
say that solutions are \emph{classical} if all derivatives appearing
in the equations and boundary conditions are continuous up to the
boundary on their domain of definition. For brevity, the problem of
looking for minimal conditions on the data, which still lead to
classical solutions, is called by us "the minimal assumptions
problem". In reference \cite{BVJDE} we considered this problem for
$2-D$ Euler equations in bounded domains. Since the study of this
classical problem leads to considering the auxiliary problem
\begin{equation}
\left\{
\begin{array}{l}
-\,\Delta\, u=\,f \quad \textrm{in} \quad \Om \,,\\
u=\,0 \quad \textrm{on} \quad \Ga\,, %
\end{array}
\right.%
\label{lapsim}
\end{equation}
it follows that the "minimal assumptions problem" for the Euler
equations led to the corresponding problem for \eqref{lapsim}. As
already remarked in section \ref{motives}, a H\"older continuity
assumption on $\,f\,$ is unnecessarily restrictive. On the other
hand, continuity of $\,f\,$ is not sufficient. This situation led us
to introduce in reference \cite{BVJDE} an "intermediate" Banach
space $\,\,C_*(\Ov)\,,$ see section \ref{introdocas} below,
$$
C^{0,\,\la}(\Ov)\Subset \,C_*(\Ov)\Subset\,C(\Ov)\,,
$$
for which, in particular, the following result holds (Theorem 4.5,
in\cite{BVJDE}).
\begin{theorem}
Let $\,f \in \,C_*(\Ov)\,$ and let $\,u\,$ be the solution of
problem \eqref{lapsim}. Then $\,u \in\, C^2(\Ov)\,,$ moreover,
\begin{equation}
\|\,\na^2\,u\,\| \leq \,c\,\|\,f\,\|_*\,.%
\label{lapili}
\end{equation}
\label{laplaces}
\end{theorem}
Contrary to H\"older's continuity case, where full regularity is
restored in the sense described in section \ref{motives}, no
significant additional regularity is obtained above, besides
continuity of $\,\na^2\,u\,.$ This set-up led us to look for
intermediate, significant, situations. We define a new family of
Banach spaces, denoted $\, D^{0,\,\al}(\Ov)\,,$ $\,\al>\,0\,,$ and
called H-log spaces, for which, the following holds. For
$\,\al>\,1>\,\bt>\,0\,$ and $\,0<\,\la\leq\,1\,,$ one has
$$
C^{0,\,\la}(\Ov)\Subset D^{0,\,\al}(\Ov) \Subset \,C_*(\Ov) \Subset
D^{0,\,\beta}(\Ov)\Subset\,C(\Ov)\,,
$$
with compact embedding. These new spaces present the following
"intermediate" regularity behavior. Consider the simplest case of
constant coefficients, second order, elliptic operators
\begin{equation}
\bL=\,\sum_1^n a_{i\,j} \pa_i\,\pa_j\,.%
\label{elle}
\end{equation}
Without loos of generality, we assume that the matrix is symmetric.
Below, we sketch a proof of the following result. A real proof is
shown merely for interior regularity, see section \ref{elipse}.
Optimality of the result is shown at the end of section \ref{one}.
\begin{theorem}
Let $\,f\in\, D^{0,\,\al}(\Ov)\,$ for some $\,\al>\,1\,,$ and let
$\,u\,$ be the solution of problem
\begin{equation}
\left\{
\begin{array}{l}
\bL\,u=\,f \quad \textrm{in} \quad \Om \,,\\
u=\,0 \quad \textrm{on} \quad \Ga\,. %
\end{array}
\right.%
\label{lapnao}
\end{equation}
Then $\,u \in\,
 D^{2,\,\al-1}(\Ov)\,,$ moreover
\begin{equation}
\|\,\na^2\,u\,\|_{\,\al-1} \leq \,C'\,\|\,f\,\|_\al\,.%
\label{hahega}
\end{equation}
The above result is optimal. If $\,\beta >\,\al-1\,,$ then
$\,\na^2\,u \in\, D^{0,\,\beta}(\Ov)\,,$ is false in general.
\label{laplolas}
\end{theorem}
The constant $\,C'\,$ depends on $\,n\,,$ $\Om\,,$ $\,\al\,$, on the
ellipticity constants of $\,\bL\,,$ and on $\,|\|\,\sg\,|\|\,$
(defined in the sequel).\par%
Since  $\,D^{0,\,\al}(\Ov) \subset\, C_*(\Ov)\,,$ for $\,\al
>\,1\,,$ it follows that (independently of the result claimed in the above theorem)
the solution $\,u\,$ of problem \eqref{lapnao} belongs to $\,C^2(\Ov)\,.$\par%

\vspace{0.2cm}

We adapt here the argument developed in the classical treatise
\cite{JBS} to prove the classical Schauder estimates, in the
framework of H\"older spaces, by appealing to the
H\"older-Korn-Lichtenstein-Giraud inequality (see \cite{JBS}, part
II, section 5.3). Following \cite{JBS}, we consider \emph{singular
kernels} $\,\mK(x)$ of the form
\begin{equation}
\mK(x)=\,\frac{\sg(x)}{|x|^n}\,,%
\label{kapes}
\end{equation}
where $\,\sg(x)\,$ is infinitely differentiable for $\,x\neq\,0\,$,
and satisfies the properties $\,\sg(t\,x)=\,\sg(x)\,,$ for
$\,t>0\,,$ and
$$
\int_S \sg(x) \,dS =\,0\,,
$$
where $\,S=\,\{\,{x:\,|x|=1}\,\}\,$. It follows easily that, for
$\,0<\,R_1<\,R_2\,,$%
\begin{equation}
\int_{R_1 <|x|< R_2} \mK(x) \,dx =\,\int_{R_1 <|x|} \mK(x) \,dx=
\,\int  \mK(x) \,dx=\,0\,,%
\label{simp}
\end{equation}
where the last integral is in the Cauchy principal value sense.\par%
For continuous functions $\,\phi\,$ with compact support, the
convolution integral, extended to the whole space $\,\R^n$,
\begin{equation}
(\mK \ast \phi)(x)=\,\int \,\mK(x-y)\,\phi(y)\,dy%
\label{convint}
\end{equation}
exists finite, as a Cauchy principal value.\par%
Set $\,I(R)=\{\,x:\,|\,x\,| \leq\, R\,\}\,,$
$\,D^{0,\,\al}(R)=\,D^{0,\,\al}(I(R))\,,$ and similar for norms and
semi-norms labeled by $R\,$. The following theorem is the H-log
counterpart of the classical result concerning H\"older spaces,
sometimes called H\"older-Korn-Lichtenstein-Giraud inequality, see
\cite{JBS}.  We show $\,D^{0,\,\al-\,1}(\Ov)\,$ regularity for
second order derivatives of potentials generated by a
$\,D^{0,\,\al}(\Ov)\,$ charge. The following result will be proved
in section \ref{apotes} below.
\begin{theorem}
Let $\,\mK(x)$ be a singular kernel, enjoying the properties
described above. Further, let $\phi \in \,D^{0,\,\al}(R)\,,$ for
some $\,\al>\,1\,,$ vanish for $\,|x| \geq\,R\,.$ Then $\,\mK \ast
\phi \in\, D^{0,\,\al-\,1}(R)\,,$ moreover
\begin{equation}
\|\,(\mK \ast \phi)\,\|_{(\al-\,1),\,R}\leq\,C\,\|\,
\phi\,\|_{\al,\,R}\,,%
\label{tima}
\end{equation}
$ C=\,c(\al)(\al-\,1)^{-\,1}\,(\|\sg\|+\,\|\,\na\,\sg\,\|\,)\,.$%
\label{ofundas}
\end{theorem}
One could also apply to ideas developed in references \cite{a-d-n},
\cite{miranda}, \cite{Sol70}, and \cite{Sol71}.
\section{On the $\,C_*(\Ov)\,$ space.}\label{introdocas}
In the context of \cite{BVJDE}, the Theorem \ref{laplaces} was
marginal. So, we did not published the proof, even though we were
not able to find it in the current literature, for boundary value
problems. At that time, we have proved the result for elliptic
boundary value problems, like \eqref{lapnao}. Our proof depend only
on the behavior of the related Green's functions, hence it applies
to larger classes of elliptic boundary value problems. Recently, by
following the same ideas, we have proved the result for the Stokes
system (see the Theorem 1.1 in \cite{BVSTOKES}):
\begin{theorem}
For every $\,\ff \in \,\bC_*(\Ov)\,$ the solution $(\bu,\,p)$ to the
Stokes system
\begin{equation}
\left\{
\begin{array}{l}
-\,\De\,\bu+\,\na\,p=\,\ff \quad \textrm{in} \quad \Om \,,\\
\na \cdot\,\bu=\,0  \quad \textrm{in} \quad \Om \,,\\
\bu=\,0 \quad \textrm{on} \quad \Ga%
\end{array}
\right.%
\label{doistokes}
\end{equation}
belongs to $\,\bC^2(\Ov)\times\,C^1(\Ov)\,$. Moreover, there is a
constant $\,c_0$, depending only on $\,\Om\,,$ such that the
estimate
\begin{equation}
\|\,\bu\,\|_2 +\,\|\na \,p\,\| \leq \,c_0\,\|\,\ff\,\|_*\,, \quad
\forall \, \ff\in\,\bC_*(\Ov)\,,%
\label{trespum}
\end{equation}
holds.%
\label{teoum-st}
\end{theorem}
It is worth noting that similar results hold for elliptic equations
with variable, sufficiently smooth, coefficients. In fact proofs
depend only on suitable properties of the Green's functions related
to the particular problem.%

\vspace{0.2cm}

For the readers convenience we recall definition and main properties
of $\,C_*(\Ov)\,$.%
For $\,f \in \,C(\Ov)\,,$ and for each $\,r>\,0\,,$ define
$\,\om_f(r) \,$ as in equation \eqref{cinco} below. In \cite{BVJDE}
we have introduced the semi-norm
\begin{equation}%
[\,f\,]_* =\,[\,f\,]_{*,\,\de} \equiv \int_0^\de \,\om_f(r) \,\frac{dr}{r}\,.%
\label{seis}
\end{equation}
The finiteness of the above integral is known as \emph{Dini's
continuity condition}, see \cite{gilbarg}, equation (4.47). The
space $\,C_*(\Ov)\,$ is defined as follows.
\begin{equation}
C_*(\Ov) \equiv\,\{\,f \in\,C(\,\Ov): \,[\,f\,]_*
<\,\infty\,\}\,.%
\label{cstar}
\end{equation}
A norm is introduced in $\,C_*(\Ov)\,$ by setting $
\,\|\,f\,\|_{*,\,\de}\equiv\,[\,f\,]_{*,\,\de}+\,\|\,f\,\|\,.$ The
following are some of the main properties of $\,C_*(\Ov)\,$ (for
complete proofs see \cite{BVSTOKES}):\par%
-- $\,C_*(\Ov)\,$ is a Banach space.\par%
-- The embedding $\, C_*(\Ov) \subset \,C(\Ov)\,$ is compact.\par%
-- The set $\,C^{\infty}(\Ov)\,$ is dense in $\,C_*(\Ov)\,.$\par%
\begin{remark}
\rm{ The results obtained in the framework of $\,C_*(\Ov)\,$ spaces
led us to also consider the problem of their possible extension to
\emph{larger} functional spaces of continuous functions. This was
done, at least partially, in references \cite{BV-LMS} and
\cite{BVALBPAO}, to which the interested reader is referred.}
\end{remark}
\section{$\,D^{0,\,\al}(\Ov)\,$ spaces: Definition and properties.}\label{one}%
The above setup led us to define and study a family of Banach spaces
$\, D^{0,\,\al}(\Ov)\,,$ $\,\al>\,0\,,$ significant in our context.
We call these spaces H-log spaces. They enjoy typical properties of
functional spaces, suitable in PDEs theory.\par%
We set%
\begin{equation}
\om_f(r) \equiv \, \sup_{\,x,\,y
\in\,\Om \,;\, 0<\,|x-\,y| \leq\,r } \,|\,f(x)-\,f(y)\,|\,.%
\label{cinco}
\end{equation}
In the sequel $\,0<\,r<\,1\,.$ For $\,\al >\,0\,,$ $\,\al
\neq\,1\,,$ one has
\begin{equation}
\int \,\frac{(-\log{r})^{-\,\al}}{r} \,dr =\,\frac{1}{\al-\,1}\,(-\log{r})^{1-\,\al}\,.%
\label{simvales}
\end{equation}
For $\,\al=\,1\,$ the right hand side of the above equation should
be replaced by $\,-\log\,(-\log{r})\,.$ It follows, in particular,
that the $\,C_*(\Ov)\,$ semi-norm \eqref{seis} is finite if
\begin{equation}
\om_f(r)\leq\, Const.\,(-\log{r})^{-\,\al}\,,%
\label{alfas}
\end{equation}
for some $\,\al >\,1\,.$ This led us to define, for each fixed
$\,\al>\,0\,,$ the semi-norm
\begin{equation}
[\,f\,]_{\al} \equiv\,\sup_{x,\,y \in\,\Ov \quad
0<\,|x-\,y|<\,1}\,\frac{|f(x)-\,f(y)\,|}{(-\log{|\,x-\,y|})^{-\,\al}}=\,
\sup_{\,r \in (0,\,1) }\,\frac{\om_f(r)}{(-\log{r})^{-\,\al}}
\,.%
\label{alfas2}
\end{equation}
The space $\,D^{0,\,\al}(\Ov)\,$ is defined as follows.
\begin{definition}
For each real positive $\,\al\,,$ we define
\begin{equation}
D^{0,\,\al}(\Ov) \equiv\,\{\,f \in\,C(\,\Ov): \,[\,f\,]_{\al}
<\,\infty\,\}\,.%
\label{cstar}
\end{equation}
A norm is introduced in $\,D^{0,\,\al}(\Ov)\,$ by setting
$$
\|\,f\,\|_{\al}\equiv\,[\,f\,]_{\al}+\,\|\,f\,\|\,.
$$
\label{defcstar}
\end{definition}
The reason for the assumption $\,|x-\,y|<\,1\,$ in \eqref{alfas2} is
due to the behavior of the function $\,\log{r}\,$ for $\,r
\geq\,1\,.$ This is clearly not restrictive from the conceptual point of view.\par%
Note that if we replace $\,0<\,|x-\,y|<\,1\,$ by
$\,0<\,|x-\,y|<\,\de_0\,,$ where $\,0<\,\de_0<\,1\,,$  and set
\begin{equation}
[\,f\,]_{\al;\,\de_0} \,\equiv\,\sup_{x,\,y \in\,\Ov \quad
0<\,|x-\,y|<\,\de_0}\,\frac{|f(x)-\,f(y)\,|}{(-\log{|\,x-\,y|})^{-\,\al}}
\,.%
\label{alfas27}
\end{equation}
then,
\begin{equation}
[\,f\,]_{\al;\,\de_0} \le\,[\,f\,]_{\al}
\leq\,[\,f\,]_{\al;\,\de_0}+\,\frac{2}{(-\log{\de_0})^{-\,\al}}\,\|\,f\,\|\,.
\label{clearnot}
\end{equation}
Hence, the norms $\,\|\,f\,\|_{\,\al}\,$ and
$\,\|\,f\,\|_{\al;\,\de_0}\,$are equivalent. We may also define, in
the canonical way, spaces $\,D^{k,\,\al}(\Ov)\,,$ for integers
$\,k\geq\,1\,.$ By writing \eqref{alfas2} in the form
$$
[\,f\,]_{\al}=\, \sup_{\,r \in (0,\,1)
}\,\om_f(r)\,(\log{\frac1r})^{\,\al}
$$
we realize that we have merely replaced in the definition of
H\"older spaces the quantity
$$
\frac{1}{r} \quad \textrm{ by}  \quad \log{\frac{1}{r}}\,,
$$
and allow $\,\al\,$ to be arbitrarily large. Finally, we may define
Banach spaces $\,D^{k,\,\al}(\Ov)\,$ endowed with the norm
$$
\|\,f\,\|_{k,\al} \equiv\,\|\,f\,\|_k +\,[\,\na^k\,f\,]_{\al}\,.
$$

\vspace{0.2cm}

Next we prove some properties of $\,D^{0,\,\al}\,$ spaces. Proofs
are similar to those followed for H\"older spaces (much simpler then
the corresponding proofs for $\,C_*(\Ov)\,$). For the reader's
convenience, without any claim of particular
originality, we adapt to our case some proofs, well known in a H\"older's
framework.\par%
It immediately follows from definitions that, for $\,\al>\,1\,,$ the
embedding $\,D^{0,\,\al}(\Ov) \subset\, C_*(\Ov)\,$ is continuous
(actually compact, see below). Furthermore,
$$
\,[\,f\,]_{*,\,\de_0} \leq\, \frac{1}{(\al-\,1)\,
(-\log{\de_0})^{\,\al-\,1}}\,[\,f\,]_{\al;\,\de_0}\,,%
$$
for fixed $\,0<\,\de_0<\,1\,.$\par%
\begin{theorem}
For each $\,\al>\,0\,,$ $\,D^{0,\,\al}(\Ov)\,$ is a Banach space.%
\label{probanas}
\end{theorem}
\begin{proof}
Let $\,f_n\,$ be a Cauchy sequence in $\,D^{0,\,\al}(\Ov)\,$. Since
$\,C(\Ov)\,$ is complete, $\,f_n\,$ is uniformly convergent to some
$\,f\in\,C(\Ov)\,.$ Let now $\,x,\,y \in \Ov\,$, with
$\,|\,x-\,y\,|<\,1\,.$ One has
$$
\frac{|f_n(x)-\,f_n(y)\,|}{(-\log{|\,x-\,y|})^{-\,\al}}\leq\,[\,f_n\,]_\al
\leq\,Const.
$$
By passing to the limit as $\,n\rightarrow\,\infty\,$ one shows that
$\,f \in \,D_{0,\,\al}\,$. On the other hand, for each arbitrary
couple $\,x,\,y \in \Ov\,$ satisfying $\,|\,x-\,y\,|<\,1\,,$ one has
\begin{equation}
\begin{array}{l}
\frac{|(f(x)-f_n(x)\,)-\,(f(y)-\,f_n(y)\,)\,|}{(-\log{|\,x-\,y|})^{-\,\al}}=\\
\\
\,\lim_{m\,\rightarrow\,\infty}
\,\frac{|(f_m(x)-f_n(x))-\,(f_m(y)-\,f_n(y)\,)\,|}{(-\log{|\,x-\,y|})^{-\,\al}}\leq\,\limsup_{m\,\rightarrow\,\infty}
\,[\,f_m-\,f_n\,]_\al\,.
\end{array}
\end{equation}
Hence,
$$
[\,f-\,f_n\,]_\al\leq\,\limsup_{m\,\rightarrow\,\infty}
\,[\,f_m-\,f_n\,]_\al\,.
$$
By appealing to the Cauchy sequence hypothesis, one shows that the
right hand side of the above equation goes to zero as
$\,n\,\rightarrow\,\infty\,.$
\end{proof}
\begin{theorem}
For $\,\al>\bt>0\,,$ the embedding $\, D^{0,\,\al}(\Ov)\subset
D^{0,\,\bt}(\Ov)\,$ is compact.%
\label{embebes}
\end{theorem}
\begin{proof}
Let  $\,f_n\,$ be a bounded sequence in  $\,D_{0,\,\al}(\Ov)\,$.
Then, by Ascoli-Arzela's theorem, there is a subsequence (denoted
again by the same symbol $\,f_n\,$), and a function $\,f \in
C(\Ov)\,,$ such that $\,f_n\,$ converges uniformly to $\,f\,.$ As in
the proof of Theorem \ref{probanas} one shows that $\,f \in
\,D_{0,\,\al}\,$. Hence, without loos of generality, by replacing
$\,f_n\,$ by $\,f_n-\,f\,,$ we may assume that $\,f=\,0\,.$
Furthermore, for$\,x,\,y \in \Ov\,$, with $\,|\,x-\,y\,|<\,1\,,$
\begin{equation}
\begin{array}{l}
\frac{|f_n(x)-\,f_n(y)\,|}{(-\log{|\,x-\,y|})^{-\,\bt}}\leq\,
\Big(\,\frac{|f_n(x)-\,f_n(y)\,|}{(-\log{|\,x-\,y|})^{-\,\al}}\,\Big)^{\frac{\bt}{\al}}\,
|f_n(x)-\,f_n(y)\,|^{1-\,\frac{\bt}{\al}}\\
\leq\,[\,f_n\,]^{\frac{\bt}{\al}}_\al\,|f_n(x)-\,f_n(y)\,|^{1-\,\frac{\bt}{\al}}
\leq\,[\,f_n\,]^{\frac{\bt}{\al}}_\al\,(\,2\,\|\,f_n\,\|\,)|^{1-\,\frac{\bt}{\al}},
\end{array}
\end{equation}
which tends to zero as $\,n\,$ goes to infinity.
\end{proof}
\begin{corollary}
The embedding $\,D^{0,\,\al}(\Ov) \subset\, C_*(\Ov)\,$ is
compact, if $\,\al >\,1\,.$%
\label{ecomes}
\end{corollary}
The result follows by appealing to the continuity of the embedding
$\,D^{0,\,\bt}(\Ov) \subset\, C_*(\Ov)\,,$ for some $\,\bt\,$
satisfying $\,\al>\,\bt>\,1\,.$\par%
\begin{theorem}
Set
$$
\Om_{\ro} \equiv\,\{\,x:\, dist(x,\,\Om\,) <\,\ro\,\}\,.
$$
There is a $\,\ro>0\,$ such that the following holds. There is a
linear continuous map $\,T\,$ from $\,C(\Ov)\,$ to
$\,C(\Ov_{\ro})\,,$ and from $\,D^{0,\,\al}(\Ov)\,$ to
$\,D^{0,\,\al}(\Ov_{\ro})\,,$ such that $\,T\,f(x)=\,f(x)\,$, for
each $\,x \in\,\Ov\,.$%
\label{bah}
\end{theorem}
\begin{proof}
For a sufficiently small positive $\,\ro\,$, which depends only on
$\,\Om\,,$ we can construct a system of parallel surfaces
$\,\Ga_r\,$, $\,0 \leq\,r\leq\,\ro\,,$ such that the surface
$\,\Ga_r\,$ lies outside $\,\Om\,,$ at a distance $\,r\,$ from
$\,\Ga_0=\,\Ga\,$. For each $\,x \in \Ov_{\ro}-\,\Om\,,$ denote by
$\,\ox\,$ the orthogonal projection of $\,x\,$ upon the boundary
$\,\Ga\,.$ We define the extension $\,T\,f=\cf\,$ by setting
$\,\cf(x) =\, f(x)\,$ if $\,x\, \in\,\Ov\,,$ and $\,\cf(x)
=\,f(\ox)\,$ if $\,x\, \in\,\Ov_{\ro} -\,\Ov\,.$ For convenience,
assume that $\,\ro<\,1\,$.\par%
Since $\,\Ga\,$ is smooth and compact, the map $\,x \rightarrow
\,\ox\,$ is Lipschitz continuous. This guarantees the existence of
the constant $\,k\,$, which depends only on $\,\Om\,,$ considered
below. The map $\,T f=\,\cf\,$ is clearly linear continuous from
$\,C(\Ov)\,$ to $\,C(\Ov_{\ro})\,$. Define
$$
\om_{\cf,\,\Om_\ro}(\,r)= \,\sup_{x,\,y
\in\,\Om_\ro ;\,|x-\,y|<\,r} \, |\,\cf(x)-\,\cf(y)\,|\,,%
$$
Next, we show that
\begin{equation}
\om_{\cf,\,\Om_\ro}(\,r) \leq\,\om_{f}(\,k\,r)\,,%
\label{olecf}
\end{equation}
for $\,r \in \,(0,\,\ro\,)\,,$ and some $\,k=\,k(\Om) \geq\,1\,$.
Assume $\,|x-\,y|\leq\,\ro\,.$ If $ x\in \,\Om_c\,$ and $ \,y
\in\,\Ov\,,$ then
$$
|\,\cf(x) -\,\cf(y)\,|=\,|\,f(\ox) -\,f(y)\,|\,,\quad \textrm{and}
\quad \,|\ox-\,y| \leq\,k\,|x-\,y|\,.
$$
If $x,\,y \in\,\Om_c\,,$ then
$$
|\,\cf(x) -\,\cf(y)\,|=\,|\,f(\ox) -\,f(\oy)\,|\,,\quad
\textrm{and}\quad |\ox-\,\oy| \leq\,k\,|x-\,y|\,.
$$
Note that, in a neighborhood of a flat portion of $\,\Ga,\,$ one has
$\,k=\,1\,.$ Equation \eqref{olecf} follows easily. Hence, for
$\,\de_0 \leq\,\ro\,,$
$$
[\,\cf\,]_{\al,\,\de_0;\,\Om_\ro} =\,
\sup_{\,0<\,r<\,\de_0}\,\frac{\om_{\cf,\,\Om_\ro}(\,r)}{(-\log{r})^{-\,\al}}
\leq\,
\sup_{\,0<\,r<\,k\,\de_0}\,\frac{\om_{f}(r)}{\Big(-\log{\frac{r}{k}}\Big)^{-\,\al}}\,.%
$$
Further, note that for $\,k>\,0\,,$
$\,0<\,r<\,k^{\frac{1}{1-\,b}}\,,$ and $\,0\,< b <\,1\,,$ one has
\begin{equation}
\log{\frac{r}{k}} \leq\,b\,\log{r}\,.%
\label{bekapa}
\end{equation}
By setting $\,b=\,\frac12\,,$ one easily shows that
$$
[\,\cf\,]_{\al,\,\de_0;\,\Om_\ro}\leq\,2^\al\,[\,f\,]_{\al,\,k\de_0}\,.
$$
Clearly, we impose to $\,\de_0\,$ the constraint $\,k\,\de_0
<\,1\,,$ to give sense to the right hand side of the above
inequality. Equivalence of full norms is guaranteed by
\eqref{clearnot}.%
\end{proof}
The next result shows that regular functions are dense in
$\,D^{0,\,\al}(\Ov)\,.$
\begin{theorem}
The set $\,C^{\infty}(\Ov)\,$ is dense in $\,D^{0,\,\al}(\Ov)\,.$%
\label{lemum}
\end{theorem}
The proof of this result, left to the reader, follows a well known
argument, namely, a suitable combination of Theorem \ref{bah} with
the classical Friedrichs' mollification technique (left to the
reader).

\vspace{0.2cm}

We end this section by showing that the regularity result claimed in
Theorem \ref{laplolas} is optimal. We assume $\,\bL=\,\De\,$. The
"singular point" is here the origin. Consider the function
\begin{equation}
u(x)=\,(\,-\,\log |x|\,)^{-\,\al}\,\sum_{i\neq\,j} x_i\,x_j\,,
\label{cotresa}
\end{equation}
where $\,\al>\,0\,,$ and $\,|x|<\,1\,.$ Direct calculations show
that each second order derivative of $\,u(x)\,$ is a combination of
negative powers of $\,-\,\log |x|\,,$ with $\,x\,$ dependent
coefficients, homogeneous of degree zero. However, the larger
negative power is $\,-\,(\al+\,1\,)\,$ for double derivatives, and
$\,-\,\al\,$ for mixed derivatives. It follows that $ \pa^{\,2}_i\,u
\in\,\,D^{0,\,\al+\,1}(\Ov)\,$, for each $i=\,1,\,...,\,n\,,$ in
particular,
$$
\De\,u \in\,\,D^{0,\,\al+\,1}(\Ov)\,.
$$
However, for $\,i\neq\,j\,,$
$$
\, \pa_i\,\pa_j\,u \notin\,D^{0,\,\beta}(\Ov)\,,
$$
if $\,\beta>\,\al\,,$ due to the presence of a $\,(\,-\,\log
|x|\,)^{-\,\al}\,$ term.\par%
By multiplication of $\,u(x)\,$ by an infinitely differentiable
function with compact support inside the unit sphere, and equal to
$\,1\,$ in a small neighborhood of the origin, we extend the above
counterexample to homogeneous boundary
value problems.\par%
The above counter-example appears formally similar to that given in
reference \cite{morrey}, at the end of section 2.6, where the author
shows that the potential of a continuous function has not
necessarily continuous second order derivatives. Note that $\,u\,$
is here the solution, not the "external force" $\,f=\,-\De\,u\,.$

\vspace{0.2cm}

Taking into account the formal relation between the above, and the
H\"older spaces, it would be interesting to define the spaces
$\,D^{0,\,\al}(\Ov)\,$ by means of an integral formulation. For instance, set%
\begin{equation}
[\,f\,]^p_{p,\,\la}\equiv\,\sup_{x_0 \in\,\Om\,,\,
0<\,\rho<\,1}\,(-\log{\rho})^{\la}
\,\int_{\Om(x_0,\,\rho)}\,|\,u(x)-\,u_{x_0,\,\rho}\,|^p
\,|\,x-\,x_0|^{-\,n} \,dx\,,%
\label{integum}
\end{equation}
where $\,u_{x_0,\,\rho}\,$ denotes the mean value of $\,u(x)\,$ in
$\,\Om(x_0,\,\rho)\,,$ and $\,\la>\,0\,.$ One easily shows that
\begin{equation}
[\,f\,]_{p,\,\la}\leq\,c\,[\,f\,]_{\al}\,, \quad \textrm{where}
\quad
\al=\,\frac{1+\,\la}{p}\,.%
\label{integdos}
\end{equation}
It should be not difficult to show that inequality \eqref{integdos}
is false for small values of the parameter $\,\al\,.$

\vspace{0.2cm}

\begin{acknowledgement}  \rm{Concerning references, the author is
particularly grateful to Francesca Crispo for calling our attention
to the treatise \cite{fiorenza}, to which the reader is referred,
and also to some related papers. In particular, as claimed in the
introduction of the above volume, spaces $\,D^{0,\,\al}(\Ov)\,,$ in
the particular case $\,\al=\,1\,,$ were introduced in reference
\cite{shara}. See also definition 2.2 in reference \cite{fiorenza}.
Other main references are \cite{diening}, \cite{samko},
\cite{zhikov1}, \cite{zhikov2}, and \cite{zhikov3}. }
\end{acknowledgement}
\section{Proof of Theorem \ref{ofundas}.}\label{apotes}%
In this section we prove the Theorem \ref{ofundas}, which is the
counterpart of the classical H\"older-Korn-Lichtenstein-Giraud
inequality in H\"older spaces.
\begin{proof}
Let $x_0,\,x_1 \in I(R)\,$, $0<|x_0-\,x_1|=\,\de <\frac19\,.$ From
\eqref{simp} it follows that
$$
(\mK \ast \phi)(x)=\,\int \,\big(\,\phi(y)-\phi(x)\big)\,\mK(x-\,y)
\,dy\,.
$$
Hence, with abbreviate notation,

\begin{equation}
\begin{array}{l}
(\mK \ast \phi)(x_0)\,-(\mK \ast \phi)(x_1)=\\
\\
\int\, \Big\{\,\big(\,\phi(y)-\phi(x_0)\big)\,\mK(x_0-\,y)
\,-\big(\,\phi(y)-\phi(x_1)\big)\,\mK(x_1-\,y)\,\Big\} \,dy=\\
\\
\int_{|y-x_0|<\,2\de} \,\{...\} \,dr +\,\int_{2\de
<|y-x_0|<\,\frac12} \,\{...\} \,dr +\,\int_{\frac12<|y-x_0|}
\,\{...\} \,dr\equiv \,I_1+I_2+I_3\,.
\end{array}
\label{decomp}
\end{equation}
Note that%
$$
\{y:\,|y-x_1|<\,2\de\} \subset\, \{|y-x_0|<\,3\de\}\,.
$$
Moreover, $3 \de < 1\,.$ So

\begin{equation}
\begin{array}{l}
\int_{|y-x_0|<\,2\de}
\,\big|\,\phi(y)-\phi(x_1)\big|\,|\mK(x_1-\,y)| \,dy \leq\\%
\\
\int_{|y-x_1|<\,3\de}
\,\big|\,\phi(y)-\phi(x_1)\big|\,|\mK(x_1-\,y)| \,dy\leq \\
\\
{[}\, \phi\,{]}_\al \int_{|y-x_1|<\,3\de} \,\frac{|\sg(x_1
-y)|}{|x_1 -y|^n}\,(-\log{|\,x_1-\,y|})^{-\,\al}\,dy\,.
\end{array}
\end{equation}
Next, by appealing to polar-spherical coordinates whith $\,
r=\,|x_1-y|\,,$  by appealing to the fact that $\sg$ is positive
homogeneous of order zero, and  by appealing to \eqref{simvales} it
readily follows that

\begin{equation}
\begin{array}{r}
\int_{|y-x_0|<\,2\de}
\,\big|\,\phi(y)-\phi(x_1)\big|\,|\mK(x_1-\,y)| \,dy \leq\\
\\
\frac{c}{\al-\,1}\,(-\log{3\de})^{-\,(\al-\,1)}\,\|\sg\|\,{[}\,
\phi\,{]}_\al\,.%
\end{array}
\label{pora3}
\end{equation}

A similar, simplified, argument shows that equation \eqref{pora3}
holds by replacing $x_1$ by $x_0$ and $3\de$ by $2\de$. So,
$$
|I_1| \leq
\,\frac{2\,c}{\al-\,1}\,(-\log{3\de})^{-\,(\al-\,1)}\,\|\sg\|\,{[}\,
\phi\,{]}_\al\,.%
$$
Since, by \eqref{bekapa}, $\,\log{(k\,\de)}\leq\,\frac12\,\log\de\,$
for $\,0<\,\de\leq\,\frac{1}{k^2}\,,$ it follows that
\begin{equation}
|I_1| \leq
\,c\,\frac{2^{\,\al}}{\al-\,1}\,(-\log{\de})^{-\,(\al-\,1)}\,\|\sg\|\,{[}\,
\phi\,{]}_\al\,.%
\label{pora4}
\end{equation}

\vspace{0.2cm}

On the other hand
$$
I_2=\,\int_{2\de <|y-x_0|<\,\frac12}
\,\big(\,\phi(x_1)-\phi(x_0)\big)\,\mK(x_0-\,y) \,dy+
$$
$$
\int_{2\de<|y-x_0|<\,\frac12}
\,\big(\,\phi(y)-\phi(x_1)\big)\,\big(\mK(x_0-\,y)-\,\mK(x_1-\,y)\big)
\,dy\,.
$$
The first integral vanishes, due to \eqref{simp}. Hence,
$$
|I_2| \leq\,\int_{2\de<|y-x_0|<\,\frac12}
\,\big|\,\phi(y)-\phi(x_1)\big|\,\big|\,\mK(x_0-\,y)-\,\mK(x_1-\,y)\,
\big| \,dy\,.
$$
Further, by the mean-value theorem, there is a point $x_2$, between
$x_0$ and $x_1$, such that
$$
\big|\,\mK(x_0-\,y)-\,\mK(x_1-\,y)\,\big|
\leq\,\big|\,\na\,\mK(x_2-\,y)\,\big|\,\de\,.
$$
Since
$$
\pa_i\,\mK(x)=\,\frac{1}{|x|^{n+\,1}}\,\Big[\,(\pa_i\,\sg)\Big(\frac{x}{|x|}\Big)\,-\,n\,\frac{x_i}{|x|}\,\sg(x)\,\Big]\,,
$$
it readily follows that
\begin{equation}
\begin{array}{l}
\big|\,\mK(x_0-\,y)-\,\mK(x_1-\,y)\,\big| \leq \\
\\
c\,\||\,\sg\,\|| \,\frac{\de}{|y-\,x_2|^{n+1}}
\leq\,c\,\||\,\sg\,\||
\,\frac{\de}{|y-\,x_0|^{n+1}}\,,%
\end{array}
\label{esao2}
\end{equation}
where $\,\||\,\sg\,\||\,$ denotes the sum of the $L^\infty$ norms of
$\sg$ and of its first order derivatives on the surface of the unit
sphere $I(0,1)\,.$ Note that
$$
|x_0-\,y|\leq\,2\,|x_2-\,y|\leq\,4\,|x_0-\,y|
$$
for
$\,|x_0-\,y|>\,2\,\de\,.$\par%
On the other hand, for $\,2\,\de <|x_0-\,y|<\,\frac12\,,$
$$
|x_1-\,y|\leq\, \frac32\,|x_0-\,y|<\,\frac34\,.
$$
So,
$$
|\,\phi(y)-\phi(x_1)\,| \leq\,[\,\phi\,]_\al \,(-\log{
(\,\frac32\,|\,x_0-\,y|}\,)\,)^{-\,\al}\,.
$$
The above estimates show that
\begin{equation}
|\,I_2\,| \leq\,C\,\||\,\sg\,\||\,[\,\phi\,]_\al
\,\de\,\int_{2\de}^{\frac12} \,
(-\log{\frac32\,r})^{-\,\al}\, r^{-2} \,dr\,.%
\label{esao2nao}
\end{equation}
Next, by appealing to L'H\^opital's rule, one gets
$$
\lim_{\de\rightarrow \,0} \,\frac{\,\int_{2\de}^{\frac12} \,
(-\log{\frac32\,r})^{-\,\al}\,r^{-2}
\,dr}{\,\de^{-\,1}(-\log{\de})^{-\,\al}}=\,\frac12\,.
$$
It follows that there is a positive constant $\de_0=\,\de_0(\al)$
such that
\begin{equation}
\de\,\int_{2\de}^{\frac12} \, (-\log{\frac32\,r})^{-\,\al}\,r^{-2}
\,dr \leq\,(-\log{\de})^{-\,\al}\,, \quad \forall \,\de  \in
\,(0,\,\de_0)\,.%
\label{dezer}
\end{equation}
We fix $\,\de_0\,$ satisfying $\de_0\leq\,\frac19\,.$ One has
\begin{equation}
|\,I_2\,| \leq\,C\,|\|\,\sg\,\||\,[\,\phi\,]_\al
\,(-\log{\de})^{-\,\al}\,,
 \quad \forall \,\de  \in
\,(0,\,\de_0)\,.%
\label{esao23}
\end{equation}

\vspace{0.2cm}

Finally we consider $I_3$. By arguing as for $I_2$, in particular by
appealing to \eqref{simp} and \eqref{esao2}, one shows that
$$
|I_3| \leq\,C\,\de\,|\|\,\sg\,\||\,\int_{|y-x_0|>\,\frac12} \,
\,\frac{|\,\phi(y)-\phi(x_1)\big|}{|y-\,x_0|^{n+1}} \,dy
\leq\,C\,\de\,|\|\,\sg\,\||\,\|\,\phi\,\|\,.
$$
This last inequality, together with \eqref{esao2nao} and
\eqref{pora4}, shows that
\begin{equation}
|\,(\mK \ast \phi)(x_0)\,-(\mK \ast
\phi)(x_1)\,|\leq\,C\,(-\log{\de})^{-\,(\al-\,1)}\,\|\,
\phi\,\|_\al\,,%
\label{tresis}
\end{equation}
for each couple of points $\,x_0,\,x_1\,$ such that
$\,0<\,|x_0-\,x_1\,| \leq\,\de_0(\al)\,.$\par%
In agreement to definition \eqref{alfas2}, we must estimate $|\,(\mK
\ast \phi)(x_0)\,-(\mK \ast \phi)(x_1)\,|\,$ for all couple of
points for which $\,0<\,|x_0-\,x_1\,| < 1\,.$ For brevity, we appeal
here to a rough argumentation. Assume that
$\,\de_0<\,|x_0-\,x_1\,|=\,\de < 1\,.$ Let $\,N>\,1\,$ be an integer
such that $\,(N-\,1)\,\de_0 <\de \leq\,\,N\,\de_0\,.$ For
simplicity, we assume that $\,\de=\,\,N\,\de_0\,.$ Consider the
sequence of points $\,z_m=\,x_0
+\,m\,\frac{\de_0}{\de}(x_1-\,x_0)\,,$ for $\,m=\,0,...,N\,.$ One
has $\,z_0=\,x_0\,,$  $\,z_N=\,x_1\,,$ and $|\,z_{m+\,1}
-\,z_m\,|=\,\de_0\,.$ By appealing to equation \eqref{tresis},
applied to each couple of points $\,z_m,\,z_{m+\,1}\,$, a standard
argument shows that
\begin{equation}
\begin{array}{l}
|\,(\mK \ast \phi)(x_0)\,-(\mK \ast
\phi)(x_1)\,|\leq\,C\,N\,(-\log{\de_0})^{-\,(\al-\,1)}\,\|\,
\phi\,\|_\al\\
\\
\leq\,\frac{C}{\de_0}\,(-\log{\de})^{-\,(\al-\,1)}\,\|\,
\phi\,\|_\al\,.%
\end{array}%
\label{esaop}
\end{equation}
\end{proof}
\section{On elliptic regularity.}\label{elipse}%
In this chapter we apply the theorem \ref{ofundas} to prove
regularity, in the framework of H-log spaces, for solutions of the
elliptic equation \eqref{lapnao}. We follow the proof developed in
H\"older spaces in \cite{JBS}, part II, section 5. So, in the
sequel, the reader will be often referred to the above reference.
We assume that $n\geq\,3\,.$%

\vspace{0.2cm}

By a fundamental solution of the differential operator $\,\bL\,$ one
means a distribution $\,J(x)\,$ in $\,\R^n\,$ such that
\begin{equation}
\bL\,J(x)=\,\de(x)\,.%
\label{fundas}
\end{equation}
The celebrated Malgrange–-Ehrenpreis theorem states that every
non-zero linear differential operator with constant coefficients has
a fundamental solution (see, for instance, \cite{yosida}, Chap. VI,
sec. 10). We recall that the analogue for differential operators
whose coefficients are polynomials (rather than constants)
is false, as shown by a famous Hans Lewy's counter-example.\par%
In particular, for a second order elliptic operator with constant
coefficients and only higher order terms, one can construct
explicitly a fundamental solution $\,J(x)$ which satisfies the
properties (i), (ii), and (iii), claimed in \cite{JBS}, namely,\par%
(i)  $J(x)$ is a real analytic function for $\,|x| \neq\,0\,.$\par%
(ii) Since $n\geq\,3\,,$
\begin{equation}
J(x)=\,\frac{\sg(x)}{|x|^{n-\,2}}\,,%
\label{jicas}
\end{equation}
where $\sg(x)$ is positive homogeneous of degree $\,0\,$.\par%
(iii) Equation \eqref{fundas} holds. In particular, for every
sufficiently regular, compact supported, function $\ph$, one has
$$
\ph(x)=\,\int \,J(x-\,y)\,(\bL\ph)(y) \,dy=\,\bL \int
\,J(x-\,y)\,\ph(y) \,dy\,.
$$
For a second order elliptic operator as above, one has%
\begin{equation}
J(x)=\,c\,\big(\,\sum\, A_{i\,j}\,x_i x_j\,\big)^{\frac{2-\,n}{2}}\,,%
\label{janota}
\end{equation}
where $A_{i\,j}$ denotes the cofactor of $a_{i\,j}$ in the
determinant $|\,a_{i\,j}\,|\,.$\par%
Following \cite{JBS}, we denote by $\bS$ the operator
\begin{equation}
(\bS\,\ph)(x)=\,\int \,J(x-\,y)\,\ph(y) \,dy=\,(J\ast\ph)(x)\,.%
\label{beesse}
\end{equation}
Note that in the constant coefficients case considered here, the
operator $\bT$ introduced in reference \cite{JBS} vanishes.\par%

Due to the structure of the function $\,\sg(x)\,$ appearing in
equation \eqref{jicas}, it readily follows that second order
derivatives of $(\bS\,\ph)(x)$ have the form $\,\pa_i\,\pa_j
\,\bS\,\ph=\,\mK_{i\,j} \ast \phi\,,$ where the $\mK_{i\,j}$ enjoy
the properties described for $\mK$ in section \ref{apotes}.\par%
We write, in abbreviate form,
\begin{equation}
\na^2\,\bS\,\ph\,(x)=\,\int \,\mK(x-y)\,\phi(y)\,dy\,,%
\label{convintos}
\end{equation}
where $\,\mK(x)\,$ enjoys the properties described at the beginning
of section \ref{apotes}.\par%
As remarked in \cite{JBS} "Lemma" A, if $\ph$ is compact supported
and sufficiently regular (for instance of class $\,C^2\,$), then
\begin{equation}
\ph=\, \bS \bL\,\ph\,,%
\label{slfi}
\end{equation}
Furthermore, if $\ph=\,\bS\,f\,,$ then $\bL\,\ph=\,f\,.$ Formally,
one has $\bS\,\bL=\,\bL \bS=\,Id\,.$

\vspace{0.2cm}

Lets prove Theorem \ref{laplolas}. Fix a no-negative $C^\infty\,$
function $\te$, defined for $\,0\leq t \leq 1\,$ such that
$\te(t)=\,1$ for $\,0\leq t \leq \frac13\,,$ and $\te(t)=\,0$ for
$\,\frac23\leq t \leq 1\,.$ Further fix a positive real $\,R\,$, for
convenience $0<R<\,\frac12\,,$ and define
\begin{equation}
\z(x)= \left\{
\begin{array}{l}
1 \quad \textrm{for} \quad |x|\leq R\,,\\
\\
\te \big(\frac{|x|-\,R}{R}\big) \quad \textrm{for} \quad R\leq
|x|\leq 2\,R\,.
\end{array}
\right.%
\label{zetaze}
\end{equation}
Let $u \in D^{2,\,\al-1}(2R)\,$ be such that $\,\bL\,u \in
D^{2,\,\al}(2R)\,,$ and set
$$
\ph=\,\z\,u\,.
$$
Note that the support of $\,\ph\,$ is contained in $\,|x|<\,2R\,.$
by appealing to \eqref{slfi}, we may write
$$
\|\na^2\,\phi\,\|_{\,\al-1;\,2R} =\,\|\na^2\, \bS \bL
\phi\,\|_{\,\al-1;\,2R}\,.
$$
From  \eqref{convintos} and Theorem \ref{ofundas}, one gets
$$
\|\,\na^2\,\bS \bL \phi\,\|_{(\al-\,1);\,2R}\leq\,C\,\|\,
\bL \,\phi\,\|_{\al;\,2R}\,.%
$$
On the other hand,
$$
 \bL \phi=\,\z  \bL u+\, N\,,
$$
where $\,\|\,N\,\|_{\al,\,2R}\leq\,C\,
\|\,\z\,\|_{2,\,\al;\,2R}\,\|\,u\,\|_{1,\,\al;\,2R} \,.$ From the
above estimates it readily follows that
$$
\|\na^2\,\phi\,\|_{\,\al-1;\,2R} \leq\,
\|\,\z\,\|_{\,\al;\,2R}\,\|\,\bL\,u\,\|_{\,\al;\,2R}+\,
C\, \|\,\z\,\|_{2,\,\al;\,2R}\,\|\,u\,\|_{1,\,\al;\,2R}\,.%
$$
So,
\begin{equation}
\|\na^2\,u\,\|_{\,\al-1;\,R} \leq\,
\|\,\z\,\|_{2,\,\al;\,2R}\,(\,\|\,\bL\,u\,\|_{\,\al;\,2R}+\,
C \,\|\,u\,\|_{1,\,\al;\,2R}\,)\,.%
\label{acasis}
\end{equation}
Next, we estimate $\|\,\z\,\|_{2,\,\al;\,2R}\,.$ Recall that that
$2R<\,1$. We consider points $|x|$ such that $R\leq\,|x|\leq\,2R\,,$
and left to the reader different situations. Moreover, due to
symmetry, it is sufficient to consider the one dimensional case
$$
\z(t)=\,\te \big(\frac{t-\,R}{R}\big) \quad \textrm{for} \quad R
\leq t\leq 2R\,.
$$
Hence
$$
\z'(t)=\,\te' \big(\frac{t-\,R}{R}\big)\,\frac{1}{R}\,,
$$
and
$$
\z''(t)=\,\te'' \big(\frac{t-\,R}{R}\big)\,\frac{1}{R^2}\,.
$$
Further,
$$
R^2\, |\z''(t_2)-\z''(t_1)| \leq\,\big| \, \te''
\big(\frac{t_2-\,R}{R}\big)
-\,\te''\big(\frac{t_1-\,R}{R}\big)\,\big|\,,
$$
where
$$
\big|\,\frac{t_2-\,R}{R}-\,\frac{t_1-\,R}{R}\,\big|=\,\big|
\frac{t_2-\,t_1}{R}\big|\leq\,\frac13 <\,1\,.
$$
So
$$
\frac{|\z''(t_2)-\z''(t_1)|}{(\,-\log |t_2-\,t_1|\,)^{-\,\al}} \leq
\,[\,\te''\,]_{H(\al)} \, \frac{1}{R^{2+\,\al}}\, \,
\frac{|t_2-\,t_1|^\al}{(\,-\log |t_2-\,t_1|\,)^{-\,\al} }\,.
$$
Since $\,0\leq \,-r\,\log{r} \leq\,1\,,$ for $\,0<\,r\,\leq \,1\,,$
one shows that
$$
[\,\z''\,]_{\,\al;\,2R} \leq\, [\,\te''\,]_{H(\al)} \,
R^{-\,(2+\,\al)}\,.
$$
It readily follows the estimate
\begin{equation}
\|\,\z\,\|_{\,2,\,\al;\,2R} \leq\, C\,\|\,\te\,\|_{H(2,\al)}
\,(\,1+\,R^{-\,1}+\,R^{-\,2}+\,R^{-\,(2+\,\al)}\,)\equiv\,C_\te(R)\,.%
\label{seguias}
\end{equation}
By appealing to \eqref{acasis}, we show that
\begin{equation}
\|\na^2\,u\,\|_{\,\al-1;\,R} \leq\,
C_\te(R)\,(\,\|\,\bL\,u\,\|_{\,\al;\,2R}+\,
C \,\|\,u\,\|_{1,\,\al;\,2R}\,)\,,%
\label{acasis-dois}
\end{equation}
for $0<2R<1\,.$

\vspace{0.2cm}

The above interior regularity result can be extended up to the
boundary by following the argument sketched in part II, section 5.6,
reference \cite{JBS}, which is essentially independent of the
particular functional space. One starts by showing that the estimate
\eqref{acasis-dois} also holds on half-spheres, under the zero
boundary condition on the flat boundary. This is achieved by means
of a "reflection" to the corresponding whole sphere, trough the flat
boundary, as an odd function on the orthogonal direction. Then,
sufficiently small neighborhoods of boundary points may be regularly
mapped, one to one, onto half-spheres. This procedure allows the
desired extension of the estimate \eqref{acasis-dois} to functions
$\,u\,$ defined on sufficiently small neighborhoods of boundary
points, vanishing on the boundary. A well known finite covering
argument lead to the thesis.


\begin{thebibliography}{35}
%
\bibitem{a-d-n}
S.~Agmon, A.~Douglis, and L.~Nirenberg, \emph{Estimates near the
boundary for solutions of elliptic partial differential equations
satisfying general boundary conditions II}, Comm. Pure Appl.
Math.,17 (1964), 35-92.
%
\bibitem{BVJDE}
H.~Beir\~ao da Veiga, \emph{On the solutions in the large of the
two-dimensional flow of a nonviscous incompressible fluid}, J. Diff.
Eq., 54, (1984), no.3, 373-389.
%
\bibitem{BVSTOKES}
H.~Beir\~ao da Veiga, Concerning the existence of classical
solutions to the Stokes system. On the minimal assumptions problem,
\emph{ J. Math. Fluid Mech.}, 16 (2014), 539-550.
%
\bibitem{BV-LMS}
H.~Beir\~ao da Veiga, \emph{An overview on classical solutions to
$2-D$ Euler equations and to elliptic boundary value problems}, in
"Recent Progress in the Theory of the Euler and Navier-Stokes
Equations", \emph{London Math. Soc. Lecture Notes}, forthcoming.
%
\bibitem{BVALBPAO}
H.~Beir\~ao da Veiga, \emph{On some regularity results for $\,2-D\,$
Euler equations and linear elliptic boundary value problems},
submitted.
%
\bibitem{JBS}
L.~Bers, F.~John, and M.~Schechter, \emph{Partial Differential
Equations}, John Wiley and Sons, Inc., New-York, 1964.
%
\bibitem{diening}
L. Diening,  \emph{Maximal function on generalized Lebesgue spaces
$\,L^{p(\cdot)}\,$ }, Math. Inequal. Appl., 7(2):245–253, 2004.
%
\bibitem{fiorenza}
D.V.~Cruz-Uribe and A.~Fiorenza, \emph{Variable Lebesgue Spaces
Foundations and Harmonic Analysis}, Springer, Basel 2013.
%
\bibitem{gilbarg}
D. Gilbarg and N.S. Trudinger, \emph{Elliptic Partial Differential
Equations of Second Order}, Springer-Verlag,
Berlin/Heidelberg/New-York, 1977.
%
\bibitem{miranda}
C.~Miranda, \emph{Equazioni alle Derivate  Parziali di Tipo
Ellittico}, Springer-Verlag, Berlin, 1955.
%
\bibitem{morrey}
C.B.~Morrey, \emph{Multiple Integrals in the Calculus of
Variations}, Springer-Verlag, Berlin-Heidelberg, 1966.
%
\bibitem{necas}
J.~Ne\v{c}as, \emph{Les Methodes Directes en Theorie des Equations
Elliptiques}, Academia, Prague, 1967.
%
\bibitem{samko}
S. Samko, \emph{Convolution type operators in $\,L^{p(x)}\,,$}
Integral Transform, Spec. Funct., 7(1–2):123– 144, 1998.
%
\bibitem{shara}
I.~I.~Sharapudinov,  \emph{The basis property of the Haar system in
the space $\,L^{p(t)}[0,\,1\,]\,$, and the principle of localization
in the mean}, Mat. Sb. (N.S.), 130(172)(2):275–283, 286, 1986.
%
\bibitem{Sol70}
V.A.~Solonnikov, \emph{On Green's Matrices for Elliptic Boundary
Problem I},  Trudy Mat. Inst. Steklov, 110 (1970), 123-170.
%
\bibitem{Sol71} V.A.~Solonnikov, \emph{On Green's Matrices for
Elliptic Boundary Problem II},  Trudy Mat. Inst. Steklov, 116
(1971), 187-226.
 %
\bibitem{yosida}
K.~Yosida, \emph{Functional Analysis}, Springer-Verlag,
Berlin-Heidelberg, 2nd edition, 1968.
%
\bibitem{zhikov1}
V. V. Zhikov, \emph{On the homogenization of nonlinear variational
problems in perforated domains}, Russian J. Math. Phys.,
2(3):393–408, 1994.
%
\bibitem{zhikov2}
V. V. Zhikov,  \emph{On Lavrentiev's phenomenon}, Russian J. Math.
Phys., 3(2):249–269, 1995.
%
\bibitem{zhikov3}
V. V. Zhikov, \emph{On some variational problems}, Russian J. Math.
Phys., 5(1):105–116 (1998), 1997.
%
\end{thebibliography}
\end{document}